\documentclass[11pt]{article}
\usepackage[utf8]{inputenc}
\usepackage{comment}
\usepackage[letterpaper,margin=1in]{geometry}
\usepackage[noblocks]{authblk}

\usepackage[utf8]{inputenc}
\usepackage{xspace}
\usepackage{enumitem}
\usepackage{amsmath,amsthm,amssymb,amsfonts,mathrsfs}
\usepackage{mathtools}
\usepackage{hyperref}
\usepackage{tikz}
\usepackage{cleveref}
\usepackage{nicefrac}
\usepackage[hyperfirst=false]{glossaries}
\setacronymstyle{long-short}

\newtheorem{theorem}{Theorem}
\newtheorem{example}[theorem]{Example}

\newtheorem{corollary}[theorem]{Corollary}
\newtheorem{lemma}[theorem]{Lemma}

\theoremstyle{definition}
\newtheorem{definition}[theorem]{Definition}

\crefname{claim}{claim}{claims}
\crefname{observation}{observation}{observations}

\newcommand{\szabo}{Szab{\'o}}
\newcommand{\SWC}{\szabo-Welzl condition\xspace}

\newcommand{\myProblem}[1]{\textsc{#1}}
\newcommand{\f}[1]{\relax\ifmmode#1\else{$#1$}\fi}

\newcommand{\USOSolTypeEndOfLine}{\emph{(U1)}\xspace}
\newcommand{\USOVioCondition}{\emph{(UV1)}\xspace}
\newcommand{\totalUSO}{\myProblem{Total-USO}\xspace}

\newacronym{USO}{\myProblem{USO-SF}}{\myProblem{Unique Sink Orientation Sink-Finding}}
\newcommand\USO{\gls*{USO}\xspace}

\newcommand{\dimension}{\relax\ifmmode{n}\else{$n$}\fi\xspace}
\newcommand{\Cube}{\relax\ifmmode{Q}\else{$Q$}\fi\xspace}
\newcommand{\orientation}{\relax\ifmmode{O}\else{$O$}\fi\xspace}
\newcommand{\nodeA}{\relax\ifmmode{v}\else{$v$}\fi\xspace}
\newcommand{\nodeB}{\relax\ifmmode{w}\else{$w$}\fi\xspace}
\newcommand\Vertices{\relax\ifmmode{V}\else{$V$}\fi\xspace}
\newcommand\Edges{\relax\ifmmode{E}\else{$E$}\fi\xspace}
\newcommand\Graph{\relax\ifmmode{G}\else{$G$}\fi\xspace}

\newacronym{PLCP}{\myProblem{P\nobreakdashes-LCP}}{\myProblem{P-Matrix Linear Complementarity Problem}}
\newcommand\PLCP{\gls*{PLCP}\xspace}

\newcommand{\totalPLCP}{\myProblem{Total-P-LCP}\xspace}

\newacronym{OMCP}{\myProblem{OMCP}}{\myProblem{Oriented Matroid Complementarity Problem}}
\newacronym{POMCP}{\myProblem{P\nobreakdashes-OMCP}}{\myProblem{P-Matroid Oriented Matroid Complementarity Problem}}
\newcommand{\OMCP}{\gls*{OMCP}\xspace}
\newcommand{\POMCP}{\gls*{POMCP}\xspace}

\newcommand{\totalPOMCP}{\myProblem{Total-P-OMCP}\xspace}

\newcommand{\Matroid}{\relax\ifmmode{{\mathcal{M}}}\else{${\mathcal{M}}$}\fi\xspace}
\newcommand{\GroundSet}{\relax\ifmmode{E}\else{$E$}\fi\xspace}
\newcommand{\Circuits}{\relax\ifmmode{\mathcal{C}}\else{$\mathcal{C}$}\fi\xspace}

\newcommand{\CircuitA}{\relax\ifmmode{X}\else{$X$}\fi\xspace}
\newcommand{\CircuitB}{\relax\ifmmode{Y}\else{$Y$}\fi\xspace}
\newcommand{\CircuitC}{\relax\ifmmode{Z}\else{$Z$}\fi\xspace}
\newcommand{\length}{\relax\ifmmode{n}\else{$n$}\fi\xspace}
\newcommand{\support}[1]{\relax\ifmmode{\underline{#1}}\else{$\underline{#1}$}\fi\xspace}
\newcommand{\edge}{\relax\ifmmode{e}\else{$e$}\fi\xspace}
\newcommand{\Basis}{\relax\ifmmode{B}\else{$B$}\fi\xspace}
\newcommand{\ExtendedMatroid}{\relax\ifmmode{\widehat{\mathcal{M}}}\else{$\widehat{\mathcal{M}}$}\fi\xspace}
\newcommand{\ExtendedGroundSet}{\f{\widehat{\GroundSet_{2\length}}}\xspace}
\newcommand{\ExtendedCircuits}{\f{\widehat{\mathcal{C}}}\xspace}

\newcommand{\Vector}{\relax\ifmmode{q}\else{$q$}\fi\xspace}
\newcommand{\SetS}{\relax\ifmmode{S}\else{$S$}\fi\xspace}
\newcommand{\sets}{\relax\ifmmode{s}\else{$s$}\fi\xspace}
\newcommand{\SetT}{\relax\ifmmode{T}\else{$T$}\fi\xspace}
\newcommand{\sett}{\relax\ifmmode{t}\else{$t$}\fi\xspace}
\newcommand{\fundamentalCircuit}{\f{C}\xspace}

\newcommand{\MatroidSolTypeEndOfLine}{\emph{(M1)}\xspace}
\newcommand{\MatroidVioPMatroid}{\emph{(MV1)}\xspace}
\newcommand{\MatroidVioPMatroidImplicit}{\emph{(MV3)}\xspace}
\newcommand{\MatroidVioNoBasis}{\emph{(MV2)}\xspace}

\newcommand{\formel}[1]{\relax\ifmmode#1\else{$#1$}\fi}

\newcommand{\Class}[1]{\textsf{#1}}
\newcommand{\TFNP}{\Class{TFNP}\xspace}
\newacronym{UEOPL}{\Class{UEOPL}}{\Class{Unique End of Potential Line}\xspace}
\newcommand\UEOPL{\gls*{UEOPL}\xspace}
\newcommand{\PUEOPL}{\Class{PromiseUEOPL}\xspace}
\newcommand\NP{\Class{NP}\xspace}
\newcommand\coNP{\Class{co-NP}\xspace}
\newcommand\FP{\Class{FP}\xspace}

\newcommand\PSPACE{\Class{PSPACE}\xspace}

\newcommand{\PM}{\left\{\begin{pmatrix}+\\ +\end{pmatrix}, \begin{pmatrix}-\\ -\end{pmatrix} \right\}}
\newcommand{\PME}{\left\{\begin{pmatrix}+\\ +\\ 0\end{pmatrix},\begin{pmatrix}-\\ -\\ 0\end{pmatrix},\begin{pmatrix} +\\ 0\\ -\end{pmatrix},\begin{pmatrix}-\\ 0\\ +\end{pmatrix},\begin{pmatrix} 0\\ +\\ +\end{pmatrix},\begin{pmatrix}0\\ -\\ -\end{pmatrix}\right\}}
\newcommand{\PMED}{\left\{\begin{pmatrix}+\\ +\\ 0\end{pmatrix},\begin{pmatrix}-\\ -\\ 0 \end{pmatrix},\begin{pmatrix} 0\\ 0\\ + \end{pmatrix},\begin{pmatrix} 0\\ 0\\ -\end{pmatrix}\right\}}

\usetikzlibrary{decorations.markings}
\usepackage{circuitikz}
\tikzset{
	->-/.style={postaction={draw=black,very thick,postaction={decorate,decoration={
					markings,
					mark=at position .5 with {\arrow{>}}
			}},}},
	-<-/.style={postaction={decorate,decoration={
				markings,
				mark=at position .5 with {\arrow{<}}
	}}},
}
\usetikzlibrary{calc}
\newcommand{\hyperplane}[3]{
\coordinate (A#1) at #2;
\coordinate (B#1) at #3;
\node(D#1) at ($ (A#1) ! 0.6cm ! 90:(B#1) $) {};
\draw[thick, shorten >=-0.75cm, shorten <=-0.75cm] (A#1) -- (B#1);
\draw[->, gray] (A#1) -- (D#1) node [above] {#1};
}

\begin{document}

\title{On Degeneracy in the P-Matroid Oriented Matroid Complementarity Problem}

\author{Michaela~Borzechowski}
    \affil{Institut f\"ur Informatik, Freie Universit\"at Berlin\\ \texttt{michaela.borzechowski@fu-berlin.de}}
\author{Simon~Weber}
    \affil{Department of Computer Science\\ ETH Zürich\\ \texttt{simon.weber@inf.ethz.ch}}

\maketitle

\begin{abstract}
Klaus showed that the \OMCP can be solved by a reduction to the problem of sink-finding in a \emph{unique sink orientation (USO)} if the input is promised to be given by a \emph{non-degenerate} extension of a \emph{P-matroid}. In this paper, we investigate the effect of degeneracy on this reduction. On the one hand, this understanding of degeneracies allows us to prove a linear lower bound on the number of vertex evaluations required for sink-finding in \emph{P-matroid USOs}, the set of USOs obtainable through Klaus' reduction. 
On the other hand, it allows us to adjust Klaus' reduction to also work with degenerate instances.
Furthermore, we introduce a total search version of the \POMCP. Given \emph{any} extension of \emph{any} oriented matroid \Matroid, by reduction to a total search version of USO sink-finding we can either solve the \OMCP, or provide a polynomial-time verifiable certificate that \Matroid is \emph{not} a P-matroid.
This places the total search version of the \POMCP in the complexity class \UEOPL.

\end{abstract}

\section{Introduction}
Degenerate input can be an issue in structural analysis and algorithm design for many algebraic and geometric problems. 
It is often swept under the rug by assuming the input to be non-degenerate. For example, one often assumes all input points of a geometric problem to be in general position. In some problems (e.g., the minimum convex partition~\cite{grelier2022minimumconvexpartition}), such an assumption is inappropriate as it makes the problem considerably easier. 
In other cases, degenerate inputs can be solved easily by resolving degeneracy using \emph{perturbation} techniques.
In this paper, we investigate degeneracy in the context of the \acrlong*{POMCP}.

In this paper we view an oriented matroid as a collection \Circuits of sign vectors, i.e., vectors over $\{-, 0, +\}$, fulfilling the so-called \emph{circuit} axioms. These sign vectors are therefore called \emph{circuits}. We will thoroughly introduce the necessary definitions in \Cref{sec:OMPrelims}.
The \acrfull*{OMCP} is a search problem; given an oriented matroid, the task is to find a circuit in \Circuits that fulfills the so-called \emph{complementarity condition} and has no negative entries. Such a circuit must of course not always exist, but Todd \cite{todd1984matroids} showed that under the promise that the input is an extension of a so-called \emph{P-matroid}, there always exists a solution, and furthermore this solution is unique. Given this promise, the problem is then known as the \acrfull*{POMCP}.
The complexity of solving a \POMCP is widely open: There are no known polynomial-time algorithms or super-polynomial lower bounds, and the problem is not known to be complete for any promise search problem complexity class.

While no polynomial-time algorithms for the \POMCP are known, Klaus \cite{klaus2012phd} provided a reduction from \emph{non-degenerate} \POMCP to the problem of finding the sink in a \emph{unique sink orientation}. Unique sink orientations are a combinatorial framework that has been introduced by  Stickney and Watson \cite{stickney1978digraph} and \szabo{} and Welzl~\cite{szabo2001usos} for solving the algebraic \PLCP. Klaus' reduction is a direct adaption of the reduction from non-degenerate \PLCP to \USO by considering ideas due to Todd~\cite{todd1984matroids}, who originally introduced the \POMCP as another combinatorial abstraction of the \PLCP.

In the \USO problem, the input succinctly encodes an orientation of the $n$-dimensional hypercube graph. It is promised that this orientation is a \emph{unique sink orientation}, i.e., an orientation in which every face (or sub-cube) contains a unique sink. The goal is to find the global sink of the cube, which --- given the promise --- exists and is unique.

The promises of both the \POMCP and \USO problem are \coNP-hard to check~\cite{gaertner2015recognizing}. Given some oriented matroid extension, it is thus infeasible to check whether Klaus' reduction is applicable, or whether the input is degenerate or not even a P-matroid extension. We fix this by extending Klaus' reduction to work between \emph{total search problem} versions of the two problems. A total search problem is obtained from a promise problem by introducing alternative solutions, so-called \emph{violation solutions}, which serve as polynomial-time checkable certificates that the promise did not hold. Luckily, such certificates exist; it is known that checking whether an orientation is a USO is \emph{in} \coNP~\cite{gaertner2015recognizing}, and we show the same for checking that a given oriented matroid extends a P-matroid. For degeneracy we do not need to introduce any violations, since even for a degenerate P-matroid extension, the \POMCP will always have a solution.

These total search problem versions of \POMCP and \USO, \totalPOMCP and \totalUSO, are members of the class \TFNP, as introduced by Megiddo and Papadimitriou~\cite{megiddo1991tfnp}. Total search problems have recently gained much research interest, and \TFNP has been refined into a complex hierarchy of subclasses like \Class{PLS} \cite{Gelb}, \Class{PPP}, \Class{PPA}, \Class{PPAD}~\cite{PPAD} and \UEOPL~\cite{UEOPL2020}. \totalPLCP and \totalUSO are known to lie in the subclass \UEOPL, which is at the bottom end of this hierarchy, i.e., it is the easiest to solve subclass apart from the class \FP of polynomial-time solvable total search problems. Similarly, \PLCP and \USO are known to lie in its promise analogue, \PUEOPL. Klaus' original reduction thus implied that non-degenerate \POMCP also lies in \PUEOPL, while our extension shows \totalPOMCP to lie in \UEOPL. None of these problems are known to be complete for their respective classes, and finding a natural complete problem for either \UEOPL or \PUEOPL remains a major open problem.

On the way towards our reduction, we develop a thorough understanding of degeneracy in \OMCP{}s and its impact on Klaus' reduction and the resulting hypercube orientations. Using this understanding and existing perturbation techniques for oriented matroids, we also derive a better lower bound for \USO on the subclasses of USOs obtainable through the reductions of Klaus~\cite{klaus2012phd} and Stickney and Watson~\cite{stickney1978digraph} from the \POMCP and from the \PLCP, respectively. Hereby we improve on the logarithmic lower bound of Weber and Widmer~\cite{weber2022matousekgap} and get closer to the almost-quadratic lower bound of Schurr and \szabo{} on general USOs~\cite{schurr2004quadraticbound}.

\subsection{Related Work}

Oriented matroids are combinatorial abstractions of many types of configurations of geometric objects, such as (pseudo-)hyperplane arrangements or point configurations.
For a thorough overview over possible sources of oriented matroids, we guide the reader to the textbook by Björner et al. \cite{bjoerner1999orientedmatroids}. P-matroids are a natural combinatorial abstraction of P-matrices, and correspondingly, the \POMCP is a combinatorial abstraction of the algebraic \PLCP. We will put our structural and lower bounds results also into the context of the \PLCP.

The \PLCP is a very important problem, since it can be used to solve many optimization and search problems, such as linear programming and strictly convex quadratic programming~\cite{gaertner2006lpuso}, simple stochastic games~\cite{gaertner2005stochasticgames,ruest2007phd}, and variants of the $\alpha$-ham sandwich problem~\cite{borzechowski2024twoisenough}. It has a similarly unknown complexity status as both \POMCP and \USO, and finding a polynomial-time algorithm for the \PLCP is a major open question. While the \PLCP reduces to both the \POMCP and \USO, and \POMCP reduces to \USO, no reductions in the other direction are known, and thus no two of these problems are known to be equivalent.

\subsection{Paper Overview}
In \Cref{sec:preliminaries} we first introduce the necessary definitions of oriented matroids and the \POMCP and \USO problems.
We then revisit Klaus' reduction in \Cref{sec:ClassicReduction} and investigate the effect of degeneracy on this reduction in \Cref{sec:structural}. In \Cref{sec:constructions} we show how degeneracies can be fixed by perturbations both in the oriented matroid as well as the USO setting. We then prove our lower bounds for sink-finding in \POMCP- and \PLCP-induced USOs (\Cref{thm:POMCPlower,thm:PLCPlower}).
In \Cref{sec:searchcomplexity} we finally show how to formulate a total search problem version of \POMCP, and adapt Klaus' reduction to the total search setting.

\section{Preliminaries}\label{sec:preliminaries}
\subsection{Oriented Matroids}\label{sec:OMPrelims}

We only introduce the definitions and notations required for our results and proofs. For a more extensive introduction to oriented matroids, we refer the reader to the comprehensive textbook by Björner et al.~\cite{bjoerner1999orientedmatroids}. We consider oriented matroids $\Matroid=(E,\Circuits)$ in \emph{circuit representation}, where $E$ is called the \emph{ground set}, and $\Circuits$ is the collection of circuits of $\Matroid$. Each circuit is a \emph{signed set} on $E$. A signed set $X$ can either be represented as a vector $\CircuitA\in\{-,0,+\}^{E}$, or as a tuple of disjoint subsets of $E$, i.e., $(\CircuitA^+, \CircuitA^-)$. From the vector representation the tuple representation can be recovered by setting $X^+:=\{e\in E\;\vert\; X_e=+\}$ and $X^-:=\{e\in E\;\vert\; X_e=-\}$, and vice versa.
We write $-\CircuitA$ for the inverse signed set $-\CircuitA =  (\CircuitA^-, \CircuitA^+)$.
The \emph{support} of a signed set $X$ is defined as the set of non-zero elements $\support{\CircuitA} \coloneqq \CircuitA^+ \cup \CircuitA^-$.

\begin{definition}[Circuit axioms]
A collection \Circuits of signed sets on a ground set \GroundSet is the collection of circuits of an oriented matroid $\Matroid = (\GroundSet, \Circuits)$ if the following set of axioms are satisfied for $\Circuits$.
\begin{itemize}[leftmargin=1.2cm]
\item[(C0)] $(\emptyset, \emptyset) \notin \Circuits$.
\item[(C1)] $X\in \Circuits \Leftrightarrow -X\in\Circuits$.
\item[(C2)] For all $\CircuitA, \CircuitB \in \Circuits$, if $\support{\CircuitA}\subseteq \support{\CircuitB}$, then $\CircuitA=\CircuitB$ or $\CircuitA=-\CircuitB$.
\item[(C3)] For all $\CircuitA, \CircuitB \in \Circuits$, $\CircuitA\neq-\CircuitB$, and $\edge \in \CircuitA^+ \cap \CircuitB^-$ there is a signed set $\CircuitC \in \Circuits$ such that
\begin{itemize}[label=$\bullet$]
\item $\CircuitC^+ \subseteq (\CircuitA^+ \cup \CircuitB^+) \setminus \{\edge\}$ and 
\item $\CircuitC^- \subseteq (\CircuitA^- \cup \CircuitB^-) \setminus \{\edge\}$.
\end{itemize}
\end{itemize}
\end{definition}

We say that a set $S\subseteq \GroundSet$ \emph{contains} a signed set $X$ if $\support{X}\subseteq S$. A \emph{basis} $\Basis \subseteq E$ of an oriented matroid $\Matroid=(\GroundSet, \Circuits)$ is an inclusion-maximal subset of \GroundSet such that \Basis contains no circuit. It is well-known that all bases of an oriented matroid \Matroid have the same size, and the \emph{rank} of \Matroid is this size. An oriented matroid is called \emph{uniform} if all subsets of $E$ with cardinality equal to the rank of the oriented matroid are bases.

The \emph{cocircuits} $\Circuits^*$ of an oriented matroid $\Matroid$ are the circuits of the \emph{dual} oriented matroid~$\Matroid^*$. To understand duality, we use the following notion of orthogonality.
\begin{definition}
    Two signed sets $X,Y$ are said to be \emph{orthogonal} if  $\support{X}\cap \support{Y}=\emptyset$, or there exist $e,f\in\support{X}\cap\support{Y}$, such that $X_eY_e=-X_fY_f$.
\end{definition}
In other words, two signed sets are orthogonal if their supports either do not intersect at all, or if they agree (same non-zero sign) and disagree (opposite non-zero sign) on at least one element.
\begin{lemma}[\cite{bjoerner1999orientedmatroids}]\label[lemma]{lem:orthogonal}
    Let $X\in \Circuits$ be a circuit and $Y\in\Circuits^*$ be a cocircuit of some oriented matroid~$\Matroid$. Then, $X$ and $Y$ are orthogonal.
\end{lemma}
 Given the set of circuits, the set of cocircuits can be computed, since the cocircuits are exactly the inclusion-minimal non-empty signed sets that are orthogonal to all circuits. Since duality of oriented matroids is self-inverse, the opposite holds too.

\begin{definition}
    In an oriented matroid $\Matroid=(E,\Circuits)$, given a basis $B$ and an element $e\not\in B$, the \emph{fundamental circuit} $C(B,e)$ is the unique circuit $\CircuitA\in\Circuits$ that fulfills $\CircuitA_e = +$ and $\support{\CircuitA}\subseteq \Basis \cup \{e \}$.
\end{definition}
\begin{definition}\label[definition]{def:fundamentalCocircuit}
    In an oriented matroid $\Matroid=(E,\Circuits)$, given a basis $B$ and an element $e\in B$, the \emph{fundamental cocircuit} $C^*(B,e)$ is the unique cocircuit $D\in\Circuits^*$ that fulfills $D_e=+$ and  $\support{D}\cap (B\setminus\{e\})=\emptyset$.
\end{definition}

An oriented matroid $\ExtendedMatroid=(E\cup\{q\},\ExtendedCircuits)$ is called an \emph{extension of }$\Matroid$ if its \emph{minor}
$\ExtendedMatroid\setminus \Vector := (\GroundSet, \{\CircuitA \mid \CircuitA \in \ExtendedCircuits \text{ and }\CircuitA_\Vector = 0 \} )$ is equal to $\Matroid$.
\begin{definition}\label[definition]{def:localization}
    Given an oriented matroid $\Matroid$ on ground set $E$ with cocircuits $\Circuits^*$, as well as an element $q\not\in E$, a function $\sigma:\Circuits^*\rightarrow\{-,0,+\}$ defines the following family of signed sets on $E\cup\{q\}$
    
    \begin{align*}
\begin{split}
\ExtendedCircuits^*:=&\{(Y,\:\sigma(Y)):Y\in\mathcal{C}^*\}\:\cup\\
&\{(Y_1\circ Y_2,\:0):Y_1,Y_2\in\mathcal{C}^*, \text{ for adjacent }Y_1,Y_2\text{ with}\\
&\;\;\sigma(Y_1)=-\sigma(Y_2)\not=0\},
\end{split}
\end{align*}
    where the notation $(X,s)$ denotes the signed set where all elements of $E$ have the same sign as in $X$, and the new element $q$ gets sign $s$. For the definitions of the ``$\circ$'' operator and of adjacency, we refer the reader to \cite{bjoerner1999orientedmatroids}, since for the further discussion, only the first of the two sets forming $\ExtendedCircuits^*$ is relevant.
    
    The function $\sigma$ is called a \emph{localization} if $\ExtendedCircuits^*$ is again a valid set of cocircuits. Then, the oriented matroid $\ExtendedMatroid$ on the ground set $E\cup\{q\}$ with cocircuits $\ExtendedCircuits^*$ is called the \emph{extension of $\Matroid$ specified by $\sigma$}.
\end{definition}

\subsection{\texorpdfstring{\POMCP}{P-OMCP}}
We consider oriented matroids $\Matroid=(E_{2n}, \Circuits)$ on the ground set $E_{2n}=S\cup T$, which is made up of two parts $S=\{s_1,\ldots,s_n\}$ and $T=\{t_1,\ldots,t_n\}$ with $S\cap T=\emptyset$. We call a set $J\subseteq E_{2n}$ \emph{complementary} if it contains no \emph{complementary pair} $s_i,t_i$.

\begin{definition}[P-matroid]\label[definition]{def:Pmatroid}
An oriented matroid $\Matroid=(E_{2n},\Circuits)$ is a \emph{P-matroid} if \SetS is a basis and \Circuits contains no sign-reversing circuit. A \emph{sign-reversing circuit} is a circuit $\CircuitA$ such that for each complementary pair $s_i,t_i$ contained in $\support{\CircuitA}$, 
$\CircuitA_{\sets_i} = -\CircuitA_{\sett_i}$.
\end{definition}

\begin{example}
\label[example]{ex:PMatroid}
Let $\GroundSet=\{s_1, t_1\}$ and $\Circuits, \Circuits'$ collections of circuits represented by sign vectors with \begin{align*}
\Circuits= \PM, \qquad \Circuits'=  \left\{\begin{pmatrix}+\\ -\end{pmatrix}, \begin{pmatrix}-\\+\end{pmatrix} \right\}.
\end{align*}
Note that in all of our examples with $n=1$, the first entry of a sign vector is the sign assigned $s_1$, and the second entry is the sign assigned to $t_1$. The matroid $\Matroid=(\GroundSet,\Circuits)$ is a P-matroid. The matroid $\Matroid'=(\GroundSet,\Circuits')$ is \emph{not} a P-matroid, since both of its circuits are sign-reversing.
\end{example}

Note that recognizing P-matroids is \coNP-hard, which is an easy to derive corollary from the fact that P-matrix recognition is \coNP-hard~\cite{P-MatrixIsCoNP-complete}.

Let \Vector be such that $\Vector \notin \GroundSet_{2\length}$.
Then $\ExtendedGroundSet \coloneqq \SetS \cup \SetT \cup \{\Vector \}$, and we write $\ExtendedMatroid = (\ExtendedGroundSet, \ExtendedCircuits)$ for an extension of $\Matroid=(\GroundSet_{2\length},\Circuits)$.
Given an extension  $\ExtendedMatroid = (\ExtendedGroundSet, \ExtendedCircuits)$, the goal of the \acrfull*{OMCP} is to find a circuit $\CircuitA \in \ExtendedCircuits$ with 
$\CircuitA^-=\emptyset$,
$\CircuitA_\Vector = +$, and
$\CircuitA_{s_i} \CircuitA_{t_i} = 0$ for every $i \in [n]$.
Note that given a simple list of all circuits in \ExtendedCircuits, finding a circuit as above can be done easily in linear time. Thus, in the \OMCP, the matroid extension is provided %
by a circuit oracle (which can be either modeled as a black-box or given as a Boolean circuit) that given a set $B\subset \ExtendedGroundSet$ and another element $e\in\ExtendedGroundSet\setminus B$ either returns that $B$ is not a basis of $\ExtendedMatroid$, or returns the fundamental circuit $C(B,e)$ (recall that this is the unique circuit $\CircuitA \in \ExtendedCircuits$ with $\CircuitA_e = +$ and $\support{\CircuitA}\subseteq \Basis \cup \{e \}$).

It is known that in P-matroids and P-matroid extensions, every complementary set $B\subset S\cup T$ of size $n$ is a basis~\cite{klaus2012phd}.
Furthermore, it is known in every P-matroid extension, the \OMCP has a unique solution~\cite{todd1984matroids}.
A P-matroid extension is called \emph{non-degenerate} if for every complementary basis $B$, the circuit $C(B,q)$ is non-zero on all elements in $B\cup\{q\}$. The \acrfull*{POMCP} is the \OMCP with the additional promise that the given matroid extension is an extension of a P-matroid.

\begin{example}\label[example]{ex:PMatroidExtension}
For \ExtendedCircuits and $\ExtendedCircuits'$ as given below, $\ExtendedMatroid=(\{s_1,t_1,q\}, \ExtendedCircuits)$ and  $\ExtendedMatroid'=(\{s_1,t_1,q\}, \ExtendedCircuits')$ are both valid extensions of the P-matroid \Matroid from \Cref{ex:PMatroid}. The third entry of each sign vector denotes the sign assigned to the new element $q$.
\Cref{fig:realizationOfPME,fig:realizationOfPMED} show the realizations of the corresponding oriented matroids as arrangements of oriented hyperplanes through the origin; each one-dimensional cell corresponds to a circuit. The red shaded areas denote the areas where $q$ is positive. The circuits marked in red are the fundamental circuits $C(\{s_1\},q)$ and $C(\{t_1\},q)$. Clearly, $\ExtendedMatroid'$ is degenerate.

The unique solution of the \POMCP instance given by \ExtendedMatroid is $\begin{pmatrix}0&+&+\end{pmatrix}^T$, the unique solution in the \POMCP given by $\ExtendedMatroid'$ is  $\begin{pmatrix}0&0&+\end{pmatrix}^T$.

\begin{figure}[h!]
\begin{minipage}{0.5\textwidth}
\centering
\begin{tikzpicture}[scale = 0.8]
\small
\filldraw[red!10] (-4,-2) -- (4,2) -- (4,3) -- (-4, 3) -- cycle;
\filldraw[black] (0,0) circle (2pt);
\node at (0, -0.5) {000};
\hyperplane{$s_1$}{(-2,2)}{(2,-2)}
\hyperplane{$t_1$}{(-3,0.5)}{(3,-0.5)}
\hyperplane{$q$}{(-3,-1.5)}{(3,1.5)}
\filldraw[red] (-2.4,2.4) circle (2pt);
\node at (-3.5,3) {\textcolor{red}{$\begin{pmatrix}0 \\ + \\ +\end{pmatrix}$}};
\filldraw[] (2.4,-2.4) circle (2pt) node[anchor=north east]{$\begin{pmatrix}0 \\ - \\ -\end{pmatrix}$};
\filldraw[red] (-3,0.5) circle (2pt);
\node at (-4.5, 0.75) {\textcolor{red}{$\begin{pmatrix}- \\ 0 \\ +\end{pmatrix}$}};
\filldraw[] (3,-0.5) circle (2pt) node[anchor=north]{$\begin{pmatrix}+ \\ 0 \\ -\end{pmatrix}$};
\filldraw[] (-3.5,-1.75) circle (2pt) node[anchor=north west]{$\begin{pmatrix}- \\ - \\ 0\end{pmatrix}$};
\filldraw[] (3,1.5) circle (2pt) node[anchor=south east]{$\begin{pmatrix}+ \\ + \\ 0\end{pmatrix}$};
\end{tikzpicture}
\end{minipage}
\begin{minipage}{0.5\textwidth}
\begin{align*}
\ExtendedCircuits= \PME
\end{align*}
\end{minipage}
\caption{Realization of \ExtendedMatroid.}
\label{fig:realizationOfPME}
\end{figure}

\begin{figure}[h!]
\begin{minipage}{0.5\textwidth}
\centering
\begin{tikzpicture}[scale = 0.8]
\small
\filldraw[red!10] (-4,-2) -- (4,2) -- (4,3) -- (-4, 3) -- cycle;
\filldraw[black] (0,0) circle (2pt);
\node at (0, -0.5) {000};
\hyperplane{$s_1$ and $t_1$}{(-3,0.5)}{(3,-0.5)}
\hyperplane{$q$}{(-3,-1.5)}{(3,1.5)}
\filldraw[red] (-3,0.5) circle (2pt);
\node at (-4.5,0.5) {\textcolor{red}{$\begin{pmatrix}0 \\ 0 \\ +\end{pmatrix}$}};
\filldraw[] (3,-0.5) circle (2pt) node[anchor=north]{$\begin{pmatrix}0 \\ 0 \\ -\end{pmatrix}$};
\filldraw[] (-3.5,-1.75) circle (2pt) node[anchor=north west]{$\begin{pmatrix}- \\ - \\ 0\end{pmatrix}$};
\filldraw[] (3,1.5) circle (2pt) node[anchor=south east]{$\begin{pmatrix}+ \\ + \\ 0\end{pmatrix}$};
\end{tikzpicture}
\end{minipage}
\begin{minipage}{0.5\textwidth}
\begin{align*}
\ExtendedCircuits= \PMED
\end{align*}
\end{minipage}
\caption{Realization of $\ExtendedMatroid'$.}
\label{fig:realizationOfPMED}
\end{figure}
\end{example}

\subsection{Unique Sink Orientation (USO)}

The \emph{$n$-dimensional hypercube graph} $Q_n$ (in short, \emph{$n$-cube}) is the undirected graph on the vertex set \mbox{$V(Q_n) = \{0,1\}^n$}, where two vertices are connected by an edge if they differ in exactly one coordinate. 
An \emph{orientation} $O\colon V(Q_n) \rightarrow \{-,+\}^n$ assigns each vertex an orientation of its incident half-edges, where $O(v)_i=+$ denotes an outgoing half-edge from vertex $v$ in dimension $i$ and $O(v)_i=-$ denotes an incoming half-edge.
A \emph{unique sink orientation (USO)} is an orientation, such that every non-empty subcube (also called a \emph{face} of the cube) contains exactly one sink, i.e., a unique vertex $v$ with $O(v)_i=-$ for all dimensions $i$ in the subcube~\cite{szabo2001usos}. Note that in a USO, each edge is oriented \emph{consistently}, i.e., the orientations of both its half-edges agree.

\begin{lemma}[\szabo{}-Welzl Condition \cite{szabo2001usos}]\label[lemma]{lem:szabowelzl}
An orientation $O$ of $Q_n$ is a USO if and only if for all pairs of distinct vertices $v,w\in V(Q_n)$, there exists a dimension $i$ such that $i$ is \emph{spanned} by $v$ and $w$ (i.e., $v_i\neq w_i$), and the half-edges in dimension $i$ incident to $v$ and $w$ are oriented in the same direction, i.e., $O(\nodeA)_i \not= O(\nodeB)_i$.
\end{lemma}

We consider the classical promise search problem \USO; given an orientation $O$ that is promised to be a USO, the task is to find the unique global sink $v$ with $O(v)_i=-$ for all $i$. Similarly to the \OMCP, simply listing all values of $O$ would take space exponential in the dimension $n$, and a sink could be found in linear time in terms of this input size. Thus, in the \USO problem the function $O$ is also given by an oracle, which can again be considered either a black-box or a Boolean circuit.

\subsection{Klaus' Reduction} \label{sec:ClassicReduction}

Klaus~\cite{klaus2012phd} showed that a non-degenerate P-matroid extension can be translated to a USO. 

\begin{lemma}[Klaus~\cite{klaus2012phd}]\label{lem:originalReduction}
The \POMCP with the promise that $\ExtendedMatroid=(\ExtendedGroundSet,\ExtendedCircuits)$ is non-degenerate can be polynomial-time reduced to the \USO problem.
\end{lemma}

In this reduction, a \POMCP on the ground set $\ExtendedGroundSet$ is turned into a USO of the $n$-cube. Every vertex~$v$ of the $n$-cube is associated with a complementary basis $B(v)\subset S\cup T$. For each $i \in [n]$, $s_i\in B(v)$ if $v_i=0$, otherwise $t_i\in B(v)$. The orientation $O(v)$ of $v$ is then computed using the fundamental circuit $C:=C(B(v),q)$.

\[
O(v)_i := \begin{cases}
+ & \text{if } C_{s_i} = - \text{ or } C_{t_i} = - ,\\
- & \text{if } C_{s_i} = + \text{ or } C_{t_i} = +.
\end{cases}
\]
As the \POMCP instance is non-degenerate, no other case can occur.
Klaus showed that the computed orientation $O$ is a USO. It is easy to see that its global sink $v$
corresponds to a fundamental circuit $C(B(v),q)$ which is positive on all elements of $B(v)\cup\{q\}$ and thus forms a solution to the \POMCP instance.

\begin{example}
Recall the P-matroid extension \ExtendedMatroid from \Cref{ex:PMatroidExtension}. \Cref{fig:USO_of_ExtendedMatroid} shows the USO created by this reduction, where $B(0)=\{s_1\}$ and $B(1)=\{t_1\}$.

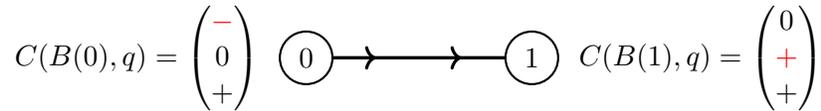
\begin{figure}[h!]
\centering
\begin{tikzpicture}[scale=1.5, roundnode/.style={circle, draw= black, thick, minimum size=7mm}]
\node[roundnode] (S) at (0,0) {$0$};
\node at (-1.5,0) {$C(B(0), q) = \begin{pmatrix} \textcolor{red}{-} \\ 0 \\ + \end{pmatrix}$};
\node[roundnode] (T) at (2,0) {$1$};
\node at (3.5,0) {$C(B(1), q) = \begin{pmatrix}0 \\ \textcolor{red}{+} \\ + \end{pmatrix}$};

\begin{scope}[very thick,decoration={markings,mark=at position 0.25 with {\arrow{>}}}] 
    \draw[postaction={decorate}] (S) -- (T);
\end{scope}

\begin{scope}[very thick,decoration={markings,mark=at position 0.75 with {\arrow{>}}}] 
    \draw[postaction={decorate}] (S) -- (T);
\end{scope}

\end{tikzpicture}
\caption{The USO created from \ExtendedMatroid  by Klaus' reduction.}
\label{fig:USO_of_ExtendedMatroid}
\end{figure}
\end{example}

\section{The Effect of Degeneracy on the Resulting USOs}\label[section]{sec:structural}
In the above reduction, if the \POMCP instance is degenerate, we sometimes cannot decide which way to orient an edge since $C_{s_i}=C_{t_i}=0$. For now, we leave these half-edges unoriented.
This leads to a \emph{partial orientation} of the hypercube, which is a function $O\colon V(Q_n) \rightarrow \{-,0,+\}^n$ where~$O(v)_i=0$ denotes an unoriented half-edge, and $O(v)_i=-$ and $O(v)_i=+$ denote an incoming and outgoing half-edge, respectively, just as in the definition of an orientation. Half-edges that are unoriented are also called \emph{degenerate}. We call a partial orientation arising from applying Klaus' reduction to a possibly degenerate P-matroid extension a \emph{partial P-matroid USO (PPU)}.
In this section we aim to understand the structure of unoriented (half-)edges in PPUs.

Not every partial orientation can be turned into a USO by directing the unoriented (half-)edges. We thus state the following condition inspired by the \szabo{}-Welzl condition.
\begin{definition}\label[definition]{def:partialSW}
    A partial orientation $O$ is said to be \emph{partially \szabo{}-Welzl} if for any two distinct vertices $v,w\in V(Q_n)$, either
    
    \begin{gather}
    O(v)_i=O(w)_i=0 \text{ for all }i \text{ such that }v_i\not=w_i,\text{ or}\\
    \text{there exists an $i$ such that } v_i\not=w_i \text{ and } \{O(v)_i,O(w)_i\}=\{-,+\}. %
    \end{gather}
\end{definition}

\begin{lemma}\label[lemma]{lem:completable}
    A partial orientation $O$ which is partially \szabo{}-Welzl can be extended to a USO.%
\end{lemma}
\begin{proof}
    We first note that in a partial orientation that is partially \szabo{}-Welzl, for every edge either both or none of the two half-edges are unoriented. By orienting all unoriented edges of $O$ from the vertex with more $1$s to the vertex with fewer $1$s (i.e., ``downwards''), any two vertices that previously fulfilled condition (1) of \Cref{def:partialSW} now fulfill the \szabo{}-Welzl condition as in \Cref{lem:szabowelzl}. Note that condition (2) of \Cref{def:partialSW} is equivalent to this classical condition on full (non-partial) orientations. We conclude that all pairs of vertices must now fulfill the \szabo{}-Welzl condition as in \Cref{lem:szabowelzl}.
\end{proof}

Note that while \Cref{lem:completable} shows that being partially \szabo{}-Welzl is sufficient for being extendable to a USO, the condition is not necessary. Borzechowski and Weber showed that deciding whether a partial orientation given by a Boolean circuit is extendable to a USO is \PSPACE-complete~\cite{borzechowski2023phases}.

\begin{lemma}\label[lemma]{lem:pmatroidszabowelzl}
    A partial P-matroid USO is partially \szabo{}-Welzl.
\end{lemma}
\begin{proof}
    Assume two vertices $v,w$ in a PPU $O$ failed both conditions of \Cref{def:partialSW}. Let $V=C(B(v),q)$ and $W=C(B(w),q)$ be the fundamental circuits used to derive $O(v)$ and $O(w)$. Since $v$ and $w$ violate the first condition of \Cref{def:partialSW}, $V\not=W$ and $V\neq-W$. Applying circuit axiom (C3) to $V$ and $-W$ to eliminate the element $q$ shows that there exists a circuit $Z$ with certain properties. Since $q\not\in\support{Z}$, $Z$ must contain both $s_i$ and $t_i$ for at least one $i\in [n]$ (since all complementary sets are independent in a P-matroid extension). As we assumed that $v$ and $w$ violate the second condition of \Cref{def:partialSW}, we know that $s_i$ and $t_i$ must have opposite signs in $Z$. Since this holds for all $i$ such that $\{s_i,t_i\}\subseteq \support{Z}$, $Z$ is a sign-reversing circuit of the underlying P-matroid, which contradicts \Cref{def:Pmatroid}. We conclude that no two vertices can fail \Cref{def:partialSW}.
\end{proof}

\Cref{lem:completable,lem:pmatroidszabowelzl} together are enough to resolve degeneracy in Klaus' reduction --- we can simply orient all unoriented edges downwards.
Thus, we can restate \Cref{lem:originalReduction} \emph{without} the promise of non-degeneracy.

\begin{corollary}
The \POMCP without a promise of non-degeneracy can be polynomial-time reduced to the \USO problem.
\end{corollary}

We observe an even stronger structural property, related to the concept of hypervertices, introduced by Schurr and \szabo{}~\cite{schurr2004quadraticbound}. A face $f$ in a (partial) orientation $O$ is called a \emph{hypervertex} if for all dimensions $i$ \emph{not} spanned by $f$, and for all vertices $v, w \in f$ it holds that $O(v)_i = O(w)_i$, i.e., all vertices in that face have the same orientation towards the rest of the cube.
Schurr and \szabo{} \cite[Corollary 6]{schurr2004quadraticbound} showed that in a USO, the orientation within each hypervertex can be replaced by an arbitrary USO without destroying the property that the orientation of the whole cube is a USO.

\begin{lemma}\label[lemma]{lem:unorientedsubcubes}
    For every partial P-matroid USO $O$ there exists a set $F$ of pairwise vertex-disjoint faces such that the set of unoriented half-edges in $O$ is exactly the union of the sets of half-edges in each $f\in F$. Furthermore, each $f\in F$ is a hypervertex.
\end{lemma}
\begin{proof}
    Let $v$ be a vertex of a PPU, and let $w$ be any other vertex within the face spanned by the unoriented edges incident to $v$, i.e., $w_i=v_i$ for all $i$ such that $O(v_i)\neq 0$. Then, the fundamental circuit $C(B(v),q)$ fulfills all the conditions that a circuit has to fulfill to be the fundamental circuit $C(B(w),q)$. Since fundamental circuits are unique in all oriented matroids~\cite{klaus2012phd}, we must have $C(B(v),q)=C(B(w),q)$ and thus $v$ and $w$ must be oriented the same way, which implies the lemma.
\end{proof}

\Cref{lem:completable,lem:pmatroidszabowelzl,lem:unorientedsubcubes}, and \cite[Corollary 6]{schurr2004quadraticbound} imply that the unoriented subcubes of a PPU can in fact be oriented according to \emph{any} USO:

\begin{corollary} \label[corollary]{cor:CompleteUSOArbitrarily}
Let $O$ be a PPU and let $O'$ be the orientation obtained by independently orienting each unoriented face $f$ of $O$ according to some USO of the same dimension as $f$. Then, $O'$ is a USO.
\end{corollary}
\begin{proof}
    By \Cref{lem:pmatroidszabowelzl,lem:completable}, $O$ can be extended to some USO. By \Cref{lem:unorientedsubcubes}, each unoriented face is a hypervertex. 
    By \cite[Corollary 6]{schurr2004quadraticbound}, each such hypervertex can be reoriented according to an arbitrary USO while preserving that the whole orientation is a USO.
\end{proof}

\begin{example}
Recall the P-matroid extension $\ExtendedMatroid'$ from \Cref{ex:PMatroidExtension}. \Cref{fig:USO_of_ExtendedMatroidDegenerate} shows the USO created by this reduction, where $B(0)=\{s_1\}$ and $B(1)=\{t_1\}$.

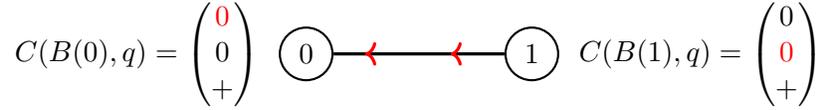
\begin{figure}[h!]
\centering
\begin{tikzpicture}[scale=1.5, roundnode/.style={circle, draw= black, thick, minimum size=7mm}]
\node[roundnode] (S) at (0,0) {$0$};
\node at (-1.5,0) {$C(B(0), q) = \begin{pmatrix}\textcolor{red}{0} \\ 0 \\ + \end{pmatrix}$};
\node[roundnode] (T) at (2,0) {$1$};
\node at (3.5,0) {$C(B(1), q) = \begin{pmatrix}0 \\ \textcolor{red}{0} \\ + \end{pmatrix}$};

\begin{scope}[very thick,decoration={markings,mark=at position 0.25 with {\arrow[red]{<}}}] 
    \draw[postaction={decorate}] (S) -- (T);
\end{scope}

\begin{scope}[decoration={markings,mark=at position 0.75 with {\arrow[red,very thick]{<}}}] 
    \path[postaction={decorate}] (S) -- (T);
\end{scope}

\end{tikzpicture}
\caption{The PPU created by the reduction from the degenerate P-matroid extension $\ExtendedMatroid'$ is first completely unoriented. After orienting every unoriented edge downwards (red) as in \Cref{lem:completable}, it becomes a USO.}
\label[figure]{fig:USO_of_ExtendedMatroidDegenerate}
\end{figure}
\end{example}

\section{Constructions Based on Degeneracy and Perturbations}\label{sec:constructions}
In this section we show how existing constructions of oriented matroid extensions can be interpreted as constructions of (partial) P-matroid USOs.
An extension $\ExtendedMatroid$ of an oriented matroid $\Matroid$ can be uniquely described by a \emph{localization}~\cite[Section 7.1]{bjoerner1999orientedmatroids}, a function $\sigma$ from the set $\Circuits^*$ of cocircuits of $\Matroid$ to the set $\{-,0,+\}$. We gave some more background about localizations in \Cref{sec:OMPrelims}. Note that not every function $f:\Circuits^*\rightarrow\{-,0,+\}$ describes a valid extension and thus not every such function is a localization.
The following lemma connects a localization to the circuits relevant to the resulting (partial) P-matroid USO.

\begin{lemma}\label[lemma]{lem:localizationsAndOrientations}
    Let $\Matroid$ be a P-matroid and let $\sigma$ be a localization for $\Matroid$ describing the extension~$\ExtendedMatroid$. Then, for any complementary basis $B$ of $\Matroid$ (and thus also of $\ExtendedMatroid$), and every element $e\in B$, the sign of $e$ in the fundamental circuit $C(B,q)$ of $\ExtendedMatroid$ is the opposite of the sign assigned by $\sigma$ to the fundamental cocircuit $C^*(B,e)$ of  $\Matroid$.
\end{lemma}
\begin{proof}
    $D:=C^*(B,e)$ is a cocircuit of $\Matroid$. By \Cref{def:localization}, $\widehat{D}:=(D,\sigma(D))$ must be a cocircuit of $\ExtendedMatroid$.
    
    By \Cref{def:fundamentalCocircuit}, $\support{\widehat{D}}$ must be a subset of $\ExtendedGroundSet\setminus (B\setminus\{e\})$ and $
    \widehat{D}_e=+$. On the other hand, the support of $C:=C(B,q)$ must be a subset of $B\cup\{q\}$, and $C_q=+$.

    \Cref{lem:orthogonal} says that $C$ and $\widehat{D}$ must be orthogonal, i.e., their supports either do not intersect, or they must agree and disagree on at least one element. Since $\support{C}\cap\support{\widehat{D}}\subseteq\{e,q\}$, the first case only occurs if $C_e=\widehat{D}_q=0$. The second case can only occur if $C_e=-\widehat{D}_q$, since $C_q=\widehat{D}_e=+$.
\end{proof}

Las Vergnas~\cite{lasvergnas1978extensions} showed that the set of localizations is closed under composition, i.e., given two localizations $\sigma_1,\sigma_2$, the following function is a localization too:

\[\forall c\in\mathcal{C}^*: (\sigma_1\circ \sigma_2)(c):=\begin{cases}
\sigma_1(c), & \text{if $\sigma_1(c)\not=0$},\\
\sigma_2(c), & \text{otherwise.}
\end{cases}\]

We now wish to understand the effect of such composition on the resulting (partial) P-matroid USO: For localizations $\sigma_1,\sigma_2$ and their corresponding PPUs $O_1,O_2$, we can see that the PPU $O'$ given by the localization $\sigma_1\circ\sigma_2$ is

\[\forall v\in V(Q_k), i\in [k]: O'(v)_i=\begin{cases}
O_1(v)_i, & \text{if } O_1(v)_i\not=0,\\
O_2(v)_i, & \text{otherwise.}
\end{cases}
\]
This follows from setting $c\in\Circuits^*$ to be $C^*(B,e)$ for $e$ being the element of $B(v)$ determining the orientation of $v$ in dimension $i$, and then applying \Cref{lem:localizationsAndOrientations}. We can see that this operation is equivalent to ``filling in'' all unoriented subcubes of $O_1$ with the orientation $O_2$.

Apart from compositions, Las Vergnas~\cite{lasvergnas1978extensions} also describes another construction technique for localizations, named lexicographic extensions.

\begin{definition}[Lexicographic extension~\cite{lasvergnas1978extensions}]\label[definition]{def:lexicographic}
    Let $\Matroid=(E,\Circuits)$ be an oriented matroid. Given an element $e\in E$ and a sign $s\in\{-,0,+\}$, the function $\sigma:\Circuits^*\rightarrow\{-,0,+\}$ given by
    
    \[\sigma(D):=\begin{cases}
    s\cdot D_e, & \text{ if } D_e\not=0,\\
    0, & \text{ otherwise,}
    \end{cases}
    \]
    is a localization. The extension of $\Matroid$ specified by this localization is called the \emph{lexicographic extension} of $\Matroid$ by $[s\cdot e]$.
\end{definition}

In the next lemma we prove that lexicographic extensions of uniform P-matroids give rise to PPUs in which the unoriented half-edges are exactly the half-edges in some facet (i.e., a face of dimension $n-1$). See \Cref{fig:lexicographicextension} for an example PPU of this form.

\begin{figure}
    \centering
    \includegraphics{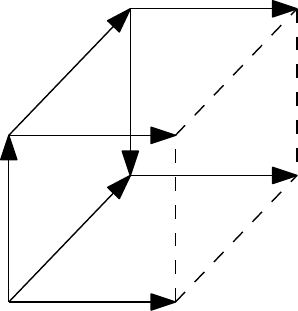}
    \caption{The form of the PPU given by a lexicographic extension of a uniform P-matroid.}
    \label{fig:lexicographicextension}
\end{figure}

\begin{lemma}\label[lemma]{lem:lexicographicextensions}
    Let $\Matroid=(E_{2n},\Circuits)$ be a P-matroid. Let $\ExtendedMatroid$ be the lexicographic extension of $\Matroid$ by $[-\cdot t_i]$. Then, in the partial P-matroid USO $O$ defined by $\ExtendedMatroid$, for any $v$ with $v_i=1$ we have $O(v)_j=0$ for $j\neq i$, but $O(v)_i=-$. In other words, the upper $i$-facet (the facet formed by the vertices $v$ with $v_i=1$) is a maximal unoriented subcube. Furthermore, this facet is a hypersink, i.e., a hypervertex with only incoming edges. Finally, if $\Matroid$ is uniform, then there are no other unoriented half-edges, i.e., $O(v)_j\neq 0$ for all $v$ with $v_i=0$ and all $j$.
\end{lemma}
\begin{proof}
    We first prove that all edges in dimension $i$ are oriented from the vertices with $v_i=0$ to the vertices with $v_i=1$.
    As $t_i$ is positive in all cocircuits $C^*(B,t_i)$ for $B$ such that $t_i\in B$ (i.e., $B=B(v)$ for $v$ with $v_i=1$), $\sigma$ assigns $-$ to all such cocircuits. By \Cref{lem:localizationsAndOrientations}, we thus have $C(B,q)_{t_i}=+$, and we see that $O(v)_i=-$ for all $v$ with $v_i=1$. Note that by \Cref{lem:pmatroidszabowelzl} we therefore have $O(w)_i=+$ for all $w$ with $w_i=0$.

    Next, we prove that all edges in the upper $i$-facet are unoriented. Let $B$ be a basis with $t_i\in B$, and let $e\in B\setminus\{t_i\}$ be some element. Now, note that $t_i\not\in\support{C^*(B,e)}$. Thus, $\sigma$ assigns $0$ to these circuits, and therefore $C(B,q)_e=0$, showing that this facet is unoriented.

    Lastly, we prove that if $\Matroid$ is uniform, the facet of vertices $v$ with $v_i=0$ is completely oriented. When $\Matroid$ is uniform, all subsets $B\subset E$ of size $n$ are bases and cobases. Thus, $|\support{C^*(B,e)}|=n+1$, and for every complementary $B$ such that $s_i\in B$ and any $e\in B$, we have that $t_i\in\support{C^*(B,e)}$ and therefore $\sigma$ assigns a non-zero sign to that circuit. This shows that $|\support{C(B,q)}|=n+1$ too, proving that all edges around a vertex $v$ with $v_i=0$ are oriented. 
\end{proof}
Of course this lemma symmetrically also applies to lexicographic extensions where~$s=+$ or $e=s_i$. Switching $t_i$ out for $s_i$ swaps the role of the two facets, and switching the sign makes the unoriented facet a hyper\emph{source} instead of a hypersink.

We can use these two construction techniques to prove a lower bound on the number of queries to the function $O$ required by deterministic sink-finding algorithms on P-matroid USOs. In essence, we successively build a localization by composition with (negative) lexicographic extensions. The construction keeps the invariant that there exists a still unqueried unoriented subcube that is guaranteed to contain the global sink. The dimension of this subcube is reduced by at most one with every query, thus at least $n$ queries are required.

\begin{theorem}\label[theorem]{thm:POMCPlower}
Let $\Matroid=(E_{2n},\Circuits)$ be a uniform P-matroid. Then, for every deterministic sink-finding algorithm $\mathcal{A}$, there exists a non-degenerate extension $\ExtendedMatroid$ of $\Matroid$ such that $\mathcal{A}$ requires at least $n$ queries to $O$ to find the sink of the P-matroid USO $O$ given by $\ExtendedMatroid$.
\end{theorem}
\begin{proof}
We specify an adversarial procedure which iteratively builds up a localization $\sigma$ for $\Matroid$ using compositions. At any point of this procedure, the current localization describes an extension $\ExtendedMatroid$ for which the PPU $O$ contains exactly one unoriented subcube $U$ that forms a hypersink, i.e., all edges incident to $U$ are oriented into $U$. Thus, the global sink of $O$ must lie in $U$, but its exact location has not been determined yet. Note that since we build up $\sigma$ using compositions, once a half-edge is oriented, its orientation never changes anymore. Thus, we only need to ensure that after a vertex $v$ is queried, $\sigma$ is extended such that $O(v)$ contains no zero entries.

At the beginning of the procedure, $\sigma$ is set to be all-zero, and thus by \Cref{lem:localizationsAndOrientations}, $O$ is completely unoriented and the unoriented subcube $U$ is the whole cube. Now, whenever the sink-finding algorithm queries a vertex $v$, we must be able to return $O(v)$ with no zero entries. If $v$ lies outside of $U$, it is already completely oriented, and $O(v)$ can simply be returned. Otherwise, if $v$ lies in $U$, the localization $\sigma$ has to be changed. To do this, we pick any dimension $i$ which spans $U$. If $v_i=0$, we change $\sigma$ to $\sigma' := \sigma \circ [-\cdot t_i]$, i.e., $\sigma$ is composed with the lexicographic extension $[-\cdot t_i]$. On $O$ this has the effect that some currently unoriented edges in $U$ become oriented. By \Cref{lem:lexicographicextensions}, in the PPU induced by $[-\cdot t_i]$ and thus also in $O$, all edges in the lower $i$-facet of $U$ are oriented, and thus $O(v)$ now contains no more zero entries. By doing this, $U$ has shrunk by one dimension (namely $i$). The invariant that $U$ is a hypersink is preserved, since the lexicographic extension we used was positive ($s=+$). Note that if instead we had $v_i=1$, we would use the lexicographic extension $[-\cdot s_i]$, with the rest of the argument staying the same.

Since $U$ shrinks by only one dimension with every query, and $U$ has $n$ dimensions at the beginning, the first $n$ vertices queried by the algorithm are never the sink. Thus, it takes at least $n$ queries to $O$ to determine the sink.
\end{proof}

Previously, the best known lower bound for sink-finding on P-matroid USOs was $\Omega(\log n)$ queries~\cite{weber2022matousekgap}. In contrast, the stronger, almost-quadratic lower bound of Schurr and \szabo{}~\cite{schurr2004quadraticbound} does not apply to P-matroid USOs (for a proof of this see \Cref{lem:schurrnonpmatroid} in \ref{sec:SchurrSzaboNonHoltKlee}).

\subsection{P-LCP} As mentioned above, the \POMCP is a combinatorial abstraction of the \acrfull*{PLCP}. In the \PLCP, one is given an $n\times n$-matrix $M$ which is promised to be a P-matrix (i.e., a matrix in which every principal minor is positive), as well as a vector $q$. The task is to find vectors $w,z$ fulfilling $w-Mz=q$, as well as the non-negativity condition $w,z\geq 0$ and the complementarity condition $w^Tz=0$. Usually, $q$ is assumed to be non-degenerate for $M$, meaning that there are no solutions to $w-Mz=q$ for which $w_i=z_i=0$ for any $i$. For a more thorough introduction to the P-LCP we refer the interested reader to the comprehensive textbook by Cottle, Pang and Stone~\cite{LCPBuch}.

It is well-known that a non-degenerate \PLCP instance $(M,q)$ can be translated to a non-degenerate \POMCP instance~\cite{klaus2012phd}. This is achieved by considering the matrix $[I; -M; -q]$, associating the first $n$ columns (those of $I$) with $S$, the next $n$ columns (those of $-M$) with $T$, and the last column with element $q$. The circuits of the P-matroid extension $\ExtendedMatroid$ are then given by the signs of the coefficients of the minimal linear dependencies of these column vectors (see \cite{bjoerner1999orientedmatroids} for a formal description of this translation of matrices to oriented matroids). Non-degenerate \PLCP{}s can thus be reduced to sink-finding in USOs (by this detour through the \POMCP, but also more directly by an equivalent reduction~\cite{stickney1978digraph}). USOs that are obtained through this reduction are called \emph{P-cubes} (forming a proper subclass of the P-matroid USOs) and have been studied intensely since faster sink-finding on P-cubes would imply algorithmic improvements for solving \PLCP{}s, linear programs, and many more optimization problems~\cite{foniok2009pivoting,foniok2014counting,fukuda2013subclass,gao2020dcubes,gaertner2006lpuso,klaus2012phd,klaus2015enumerationplcp,morris2002distinguishing,morris2002pivotalgo,ruest2007phd,weber2021matousek,weber2022matousekgap}. Degenerate \PLCP{}s yield degenerate \POMCP{}s, and we call their corresponding partial orientations \emph{partial P-cubes}.

All of our results in \Cref{sec:structural,sec:constructions} also hold in the context of \PLCP{}s. The structural results from \Cref{sec:structural} naturally hold for partial P-cubes, since they are a (proper) subset of the partial P-matroid USOs. The constructions from \Cref{sec:constructions} can also be translated to work on \PLCP{}s. Different extensions of a P-matroid can be seen as different vectors $q$ for the same P-matrix $M$. The composition of localizations can be replaced by the composition of vectors $q$, computed as $q':=q_1+\epsilon\cdot q_2$ for some $\epsilon>0$ chosen to be small enough that no sign of a non-zero element of any candidate solution (i.e., a pair of vectors $w,z$ fulfilling only $w-Mz=q$) is flipped. The lexicographic extensions correspond to setting $q$ to be equal to vectors $e_i,-e_i,M_i,$ or $-M_i$. Armed with these constructions equivalent to both compositions and lexicographic extensions, we can restate \Cref{thm:POMCPlower} in the context of the \PLCP:
\begin{theorem}\label{thm:PLCPlower}
Let $M$ be an $n\times n$ P-matrix such that the matrix $[I; -M]$ has no linear dependencies of fewer than $n+1$ columns. Then, for every deterministic sink-finding algorithm $\mathcal{A}$, there exists a vector $q$ such that $\mathcal{A}$ requires at least $n$ queries to find the sink of the P-cube given by the \PLCP instance $(M,q)$.
\end{theorem}

\section{The Search Problem Complexity of \texorpdfstring{\POMCP}{P-OMCP}}\label{sec:searchcomplexity}

\POMCP and \USO as we considered them in the previous sections are promise search problems.
In general, we would like to get meaningful output even if the promise does not hold.
Unfortunately, as discussed in \Cref{sec:preliminaries}, we cannot efficiently decide whether the promises hold, since they are both \coNP-hard to check. 
Therefore, we define total search versions for both problems.

A total search problem is a search problem for which a solution always exists.
The complexity class \TFNP (Total Function Problems in \NP), introduced by Megiddo and Papadimitriou~\cite{megiddo1991tfnp}, captures such total search problems for which any candidate solution is verifiable in polynomial time.

In order to make our promise problems total, we introduce \emph{violation solutions}, exploiting the simple falsifiability (i.e., \coNP-containment) of the promises. A violation solution is a polynomial-time verifiable certificate proving that the promise does not hold. Given an input instance, we then wish to either find a valid solution (i.e., a solution of the type we search for in the promise problem), or a violation solution if the promise does not hold. Thus, there always exists a solution and the problem is total. We would like to note that even if the promise does not hold a valid solution may exist and may be returned to solve the total search problem.

\begin{definition}[{\totalUSO}]
Given a boolean circuit computing an orientation function $\orientation\colon \{0, 1\}^\dimension \rightarrow \{+, -\}^\dimension$ of a hypercube, find one of the following.

\begin{itemize}[leftmargin=1.2cm]
\item[\USOSolTypeEndOfLine] A sink, i.e., a vertex \f{\nodeA \in \{0, 1\}^\dimension} such that \f{\orientation(\nodeA)_i = -} for all $i\in[n]$.

\item[\USOVioCondition] Two distinct vertices \f{\nodeA, \nodeB \in \{0, 1\}^\dimension} with either
$\nodeA_i = \nodeB_i$ or $O(\nodeA)_i = O(\nodeB)_i$ for all $i\in[n]$, showing that the orientation \orientation does not fulfill the \SWC and thus is not USO. 
\end{itemize}
\end{definition}

Next, we wish to restate \POMCP as total search problem. 
Observe that a simple polynomial-time verifiable certificate for a given oriented matroid not being a P-matroid is simply a circuit $Z \in \Circuits$ which is sign-reversing.
Nonetheless, we define \totalPOMCP with more than one violation solution type. 
This is due to the fact that we want to reduce \totalPOMCP to \totalUSO, and the reduction may only detect violations of type \MatroidVioNoBasis and \MatroidVioPMatroidImplicit.
We do not know of an efficent way of transforming a violation of one of these types to a violation of type \MatroidVioPMatroid.
Note that as Fearnley et al. \cite{UEOPL2020} pointed out, there may be a difference in the complexity of a total search problem depending on the violations chosen.

\begin{definition}[{\totalPOMCP}]\label[definition]{def:TotalPOMCP}
Let  $\ExtendedMatroid = (\ExtendedGroundSet, \ExtendedCircuits)$ be an oriented matroid.
The task of the total search problem \totalPOMCP is to find one of the following.
\begin{itemize}[leftmargin=1.2cm]
\item[\MatroidSolTypeEndOfLine] A circuit $\CircuitA \in \ExtendedCircuits$ such that
$\CircuitA^-=\emptyset$,
$\CircuitA_\Vector = +$ and
$\forall i \in [n] : \CircuitA_{s_i} \CircuitA_{t_i} = 0.$

\item[\MatroidVioPMatroid] A circuit $\CircuitC \in \Circuits$ which is sign-reversing.

\item[\MatroidVioNoBasis] A complementary set $B \subset E_{2n}$ of size $n$ which is not a basis of \ExtendedMatroid.

\item[\MatroidVioPMatroidImplicit] Two distinct, complementary circuits $\CircuitA, \CircuitB \in \Circuits$ with $\CircuitA_q = \CircuitB_q = +$ such that for all $i \in [n]$:
\begin{itemize}[ label=$\bullet$]
    \item $X_{s_i}Y_{t_i}=X_{t_i}Y_{s_i}=0$ or 
    \item $X_{s_i}=Y_{t_i}$ and $X_{t_i}=Y_{s_i}$.
\end{itemize}
\end{itemize}
\end{definition}

Clearly, a solution of type \MatroidVioPMatroid or \MatroidVioNoBasis implies that \ExtendedMatroid is not a P-matroid extension by definition and by the fact that every complementary set of size $n$ is a basis in a P-matroid extension~\cite{klaus2012phd}. The definition of the violation \MatroidVioPMatroidImplicit on the other hand may look rather unintuitive, but the following lemma shows that it correctly implies that $\ExtendedMatroid$ is not a P-matroid extension.

\begin{lemma}\label[lemma]{lem:VioImpliesNotPMatroid}
A violation of type \MatroidVioPMatroidImplicit implies that $\ExtendedMatroid$ is not a P-matroid extension.   
\end{lemma}

\begin{proof}
Suppose we are given such a violation, i.e., two distinct complementary circuits $\CircuitA, \CircuitB \in \Circuits$ with $\CircuitA_q = \CircuitB_q = +$ and $\forall i \in [n] : X_{s_i}Y_{t_i}=X_{t_i}Y_{s_i}=0$ or $X_{s_i}=Y_{t_i}$ and $X_{t_i}=Y_{s_i}$.

As $X,Y$ are distinct, $X\neq Y$. Since $X_q=Y_q=+$, it holds that $X\neq -Y$. We can thus apply circuit axiom (C3) on circuits $X$ and $-Y$ and element $q\in X^+\cap (-Y)^-$.
It follows that there must exist some circuit \CircuitC with: 
\begin{itemize}
	\item $\CircuitC^+ \subseteq X^+ \cup (-Y)^+ \setminus \{q\}$ and 
	\item $\CircuitC^- \subseteq X^- \cup (-Y)^- \setminus \{q\}$.
\end{itemize}
Suppose $\support{Z}$ contained no complementary pair, then it is a complementary set. Any complementary set $B\supseteq\support{Z}$ of size $n$ cannot be a basis, since $\support{Z}$ is a circuit. This is a violation of type \MatroidVioNoBasis and implies that $\ExtendedMatroid$ is not a P-matroid extension.

Thus, $\support{Z}$ must contain at least one complementary pair $s_i,t_i$. 
As $X$ and $Y$ are complementary, $s_i$ and $t_i$ are each only contained in one of the two circuits, w.l.o.g. $s_i\in\support{X}$ and $t_i\in\support{Y}$. Therefore, $s_i$ and $t_i$ are each only contained in one of the two sets \mbox{$X^+\cup (-Y)^+\setminus\{q\}$} and $X^-\cup(-Y)^-\setminus\{q\}$. Since $X_{s_i}=Y_{t_i}$, they are both in different sets, and thus $Z_{s_i}=-Z_{t_i}$.
Since this holds for every complementary pair in $\support{Z}$, we conclude that $Z$ is sign-reversing. Thus $Z$ is a violation of type \MatroidVioPMatroid, and $\ExtendedMatroid$ can not be a P-matroid extension.

Note that even if we cannot find \CircuitC explicitly in polynomial time, we can check the conditions on \CircuitA and \CircuitB in polynomial time.
\end{proof}

With the help of \Cref{lem:completable,lem:pmatroidszabowelzl} we now adapt Klaus' reduction of non-degenerate \POMCP instances to \USO (recall \Cref{sec:ClassicReduction}) to also work with degenerate instances and their respective total search versions.

Given a \totalPOMCP instance $\ExtendedMatroid = (\ExtendedGroundSet, \ExtendedCircuits)$ (note that $\ExtendedMatroid$ is possibly not a P-matroid extension, or degenerate), 
we associate every vertex $v$ of the cube with a complementary basis $B(v)\subset S\cup T$. 
If $B(v)$ is not even a basis due to \ExtendedMatroid not being a P-matroid extension, we simply immediately make $v$ a sink.
Otherwise, for each $i \in [n]$ we have $s_i\in B(v)$ if $v_i=0$, and  $t_i\in B(v)$ if $v_i=1$. The orientation $\orientation\colon V(Q_n) \rightarrow \{+, -\}^n$ is then computed using the fundamental circuit $C:=C(B(v),q)$:

\begin{align}\label{eq:orientation}
\orientation(v)_i := \begin{cases}
\begin{rcases}
- & \makebox[0pt][l]{if $B(v)$ is not a basis,}\phantom{aaaaaaaaaaaaaaaaaaaaaa}  \\
\end{rcases} & \text{Case (1): handle missing basis.}  \\
 \begin{rcases}
- & \makebox[0pt][l]{if $v_i=0$ and $\fundamentalCircuit_{s_i} = 0$,}\phantom{aaaaaaaaaaaaaaaaaaaaaa} \\
+ & \text{if } v_i=1 \text{ and }\fundamentalCircuit_{t_i} = 0 , \\
\end{rcases}
&\text{Case (2): orient degenerate half-edge.} \\
\begin{rcases}
+ & \text{if } v_i=0 \text{ and } \fundamentalCircuit_{s_i} = - \\
  & \makebox[0pt][l]{or $v_i=1$ and $\fundamentalCircuit_{t_i} = - $,}\phantom{aaaaaaaaaaaaaaaaaaaaaa}\\
- & \text{if } v_i=0 \text{ and } \fundamentalCircuit_{s_i} = + \\
  & \text{or } v_i=1 \text{ and }\fundamentalCircuit_{t_i} = + .\\
\end{rcases} & \text{Case (3): orient a non-degenerate half-edge.}
\end{cases}
\end{align}

\begin{theorem} \label{thm:ReductionIsPromisePreserving}
The construction above is a polynomial-time reduction from \totalPOMCP to \totalUSO.
\end{theorem}
\begin{proof}
Given a \totalPOMCP instance $\ExtendedMatroid = (\ExtendedGroundSet, \ExtendedCircuits)$, let $(Q_n, O)$ be an \totalUSO instance with \orientation as defined above.

\paragraph*{Polynomial time}
For the reduction we build an orientation oracle $O$ for \totalUSO from the given circuit oracle for \totalPOMCP. Note that this does not mean that we have to compute the output of $O$ for every vertex, we simply have to build the circuit computing $O$ from the circuit computing the oriented matroid circuit oracle.

Since $O$ merely computes $B(v)$ from a given vertex, invokes the circuit oracle, and then performs a case distinction, it can clearly be built and queried in time polynomial in $n$.

\paragraph*{Correctness}
To prove correctness of this reduction, we must show that 
every valid solution or violation solution of \totalUSO can be mapped back to a valid solution or  violation solution of \totalPOMCP in polynomial time. Since our reduction always creates a valid USO if \ExtendedMatroid is a P-matroid extension, we will map back violation solutions of \totalUSO only to violation solutions of \totalPOMCP.

\paragraph*{A solution of type \USOSolTypeEndOfLine}

Let $\nodeA \in V(Q_n)$ be a solution to the \totalUSO instance, i.e., a sink. 

It might be that $v$ is a sink because $B(v)$ is not a basis, and thus $O(v)_i = -$ for all $i$. To map this back to a violation or solution of \totalPOMCP, we first check if the \totalPOMCP oracle returns that $B(v)$ is not a basis for the input $\fundamentalCircuit(B(v), q)$. If so, we found a violation of type \MatroidVioNoBasis.
Otherwise, the fundamental circuit $\fundamentalCircuit(B(v), q)$ is a solution to the \totalPOMCP instance: Since \nodeA is a sink, there is no index at which the fundamental circuit is negative. All entries of $\fundamentalCircuit(B(v), q)$ are positive and the complementarity condition is fulfilled by construction of $B(v)$. 

\paragraph*{A violation of type \USOVioCondition}

If a violation is found, we have two distinct vertices \nodeA and \nodeB such that for all
$i\in [n]$, either $\nodeA_i = \nodeB_i$ or $O(\nodeA)_i = O(\nodeB)_i$. We first again check whether $B(v)$ and $B(w)$ are bases, if not we map this violation to a violation of type \MatroidVioNoBasis.

Otherwise, we show that there are two distinct complementary circuits 
$\CircuitA, \CircuitB \in \Circuits$ with $\CircuitA_q = \CircuitB_q = +$ and $X_{s_i}Y_{t_i}=X_{t_i}Y_{s_i}=0$ or $X_{s_i}=Y_{t_i}$ and $X_{t_i}=Y_{s_i}$ for all $i$, i.e., a violation of type \MatroidVioPMatroidImplicit.
We claim that the circuits
$\CircuitA := \fundamentalCircuit(\Basis(\nodeA), q)$ and $\CircuitB :=\fundamentalCircuit(\Basis(\nodeB), q)$ fulfill these conditions.

First, we need to show that $\CircuitA\neq \CircuitB$. If the two circuits were equal, they would have to be degenerate on all dimensions spanned by $v$ and $w$. Then by construction of $O$, $v$ and $w$ could not fail the \szabo{}-Welzl condition (see \Cref{lem:completable}).

Next, we observe that by definition we have $\fundamentalCircuit(\Basis(\nodeA), q)_q = +$ and $\fundamentalCircuit(\Basis(\nodeB), q)_q=+$ and both circuits are complementary.

Finally, we show that for each dimension $i$, either (i) $X_{s_i}Y_{t_i}=X_{t_i}Y_{s_i}=0$ or (ii) $X_{s_i}=Y_{t_i}$ and $X_{t_i}=Y_{s_i}$. For every dimension $i$ for which $v_i=w_i$ (w.l.o.g. both are $0$), both $\CircuitA_{t_i}=0$ and $\CircuitB_{t_i}=0$. Therefore, condition (i) holds. For a dimension $i$ for which $v_i\neq w_i$, if at least one of the $i$-edges incident to $v_i$ and $w_i$ is degenerate, we have $X_{s_i}=X_{t_i}=0$ (or $Y_{s_i}=Y_{t_i}=0$). Thus, condition (i) also holds in this case.
For a dimension $i$ in which both are non-degenerate, since $v$ and $w$ are a violation of type \USOVioCondition, $O(v)_i=O(w)_i$. By construction of $O$ it must hold that $X_{s_i}=Y_{t_i}$ and $X_{t_i}=Y_{s_i}$, i.e., condition (ii) holds.

Therefore, the circuits $\fundamentalCircuit(\Basis(\nodeA), q)$ and $\fundamentalCircuit(\Basis(\nodeB), q)$ form a violation of type \MatroidVioPMatroidImplicit.
\end{proof}

Since \totalUSO is in the complexity class $\UEOPL\subseteq \Class{PPAD} \cap \Class{PLS}\subseteq \Class{TFNP}$, we now also get the following corollary.
\begin{corollary}
\totalPOMCP is in \UEOPL.
\end{corollary}

Note that neither of the two problems are known to be \UEOPL-hard.

\newpage
\bibliography{USO.bib,UEOPL.bib,additional_lit.bib}
\newpage

\appendix

\section{Schurr and \szabo{}'s Lower Bound}\label{sec:SchurrSzaboNonHoltKlee}
On an intuitive level, Schurr and \szabo{}'s adversarial construction yielding the $\Omega(n^2/\log n)$ lower bound for deterministic sink-finding algorithms~\cite{schurr2004quadraticbound} works as follows:
In a first phase, the construction answers $n-\lceil\log_2 n\rceil$ queries of the algorithm. After these queries, it is guaranteed that there exists some face of dimension $n-\lceil\log_2 n\rceil$ in which no vertex has been queried yet. The queries in the first phase are answered such that this face is a hypersink (all edges are incoming) and can thus be filled in with any USO. The lower bound then follows from a recursive argument.

In more detail, in the first phase, the construction keeps a set of dimensions $L$ and a USO~$\tilde{s}$ of the cube spanned by the dimensions in $L$. Queried vertices are oriented according to their projected location in $\tilde{s}$ for the dimensions in $L$, and always outgoing for the dimensions in $[n]\setminus L$. As an invariant, no two queried vertices can be the same vertex when projected onto the cube spanned by $L$, thus a dimension is added to $L$ (and $\tilde{s}$ is adapted by a defined procedure) whenever this condition would be violated.

\begin{lemma}\label[lemma]{lem:schurrnonpmatroid}
There exists a deterministic sink-finding algorithm $\mathcal{A}$, against which Schurr and \szabo{}'s adversarial construction from the proof of \cite[Theorem~9]{schurr2004quadraticbound} produces a USO that is not a P-matroid USO.
\end{lemma}
\begin{proof}
We describe an algorithm $\mathcal{A}$ that forces the construction to make $\tilde{s}$ a fixed $3$\nobreakdash-dim\-ensional subcube which is not a P-matroid USO. For this strategy to work, we require five queries. We set the dimension of the final cube to be at least $8$, such that the construction stays in the first phase for at least five queries. 

All vertices queried by our algorithm have a zero in all coordinates except the first three; we therefore omit writing these additional zeroes in their coordinates. The algorithm begins by querying the following two vertices:

\[v_1=000,\;v_2=111\]
After the second query, the construction has to add one of the first three dimensions to the set $L$, since otherwise $v_1$ and $v_2$ have the same coordinates within the (empty) set $L$. Note that the algorithm can detect this choice $\ell$, as the only incoming edge of $v_2$ is in this dimension $\ell$. W.l.o.g., we assume the algorithm picks $\ell=1$, i.e., $L=\{1\}$.

The algorithm continues by querying $v_3=011$. Once again, the construction has to pick either the second or third dimension to be added to $L$, as otherwise $v_1$ and $v_3$ have the same coordinates within the dimensions in $L$. Again this choice can be detected by the algorithm, and w.l.o.g. we assume that now $L=\{1,2\}$.

The algorithm now queries $v_4=100$, and $L$ does not change, since all vertices $v_1,\ldots,v_4$ have different coordinates in the dimensions $\{1,2\}$.

The final query is $v_5=001$, and now $L$ must be changed to $\{1,2,3\}$. The USO $\tilde{s}$ on the cube spanned by $L$ evolves with these queries as shown in \Cref{fig:schurrnonholtklee}. Note that at no point the construction has any choice in how to orient the edges in the newly added dimension, since all edges of dimensions not in $L$ incident to queried vertices must be oriented away from the queried vertex by definition of the construction. Thus, the construction is forced to build this orientation when confronted with our algorithm.

\begin{figure}[htb]
\centering
\includegraphics[keepaspectratio,width=0.86\columnwidth]{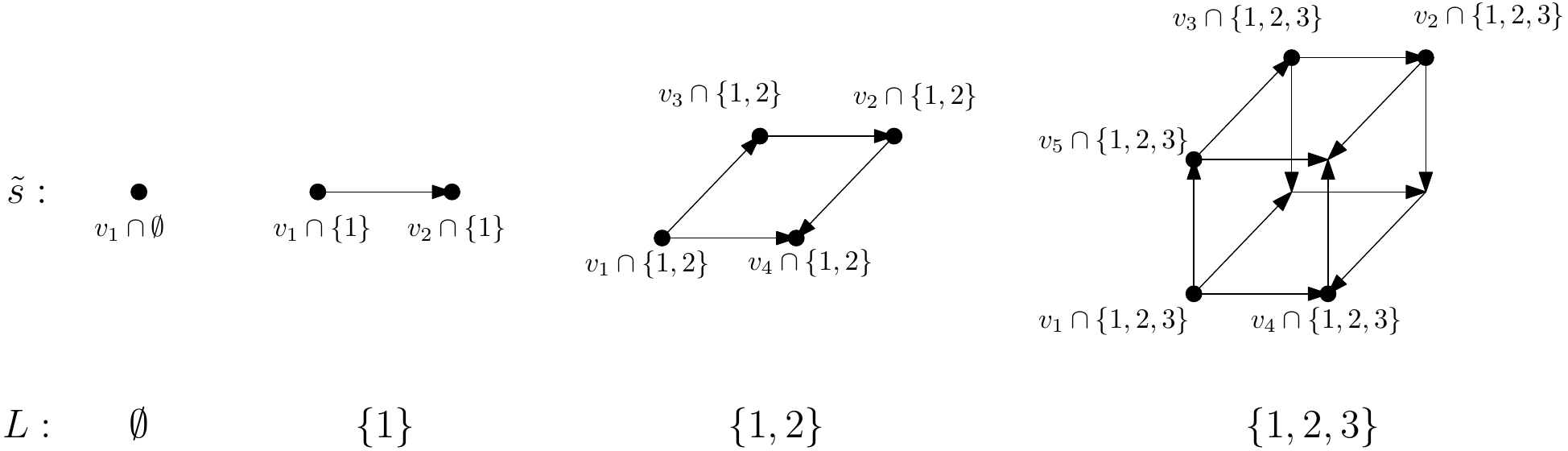}
\caption{The orientation $\tilde{s}$ built by Schurr and \szabo{}'s adversarial construction. The notation $v\cap L$ describes projection of $v$ onto the cube spanned by the dimensions in $L$.}\label{fig:schurrnonholtklee}
\end{figure}

By \cite{gaertner2008grids}, every $n$-dimensional P-cube must contain $n$ vertex-disjoint paths from the source to the sink (a property called \emph{Holt-Klee}). The orientation $\tilde{s}$ clearly does not fulfill this, as it does not have three edges from the lower $3$-facet (containing the source) to the upper $3$-facet (containing the sink). Thus, $\tilde{s}$ is not a P-cube. In $3$ dimensions, the set of P-cubes is the same as the set of P-matroid USOs~\cite{klaus2012phd} (this follows from every oriented matroid of $7$ elements being realizable). We thus also know that $\tilde{s}$ is not a P-matroid USO. As both P-cubes and P-matroid USOs are closed under the operation of taking subcubes~\cite{klaus2012phd}, and because the final orientation constructed by the construction of Schurr and \szabo{} contains $\tilde{s}$ as the subcube spanned by $v_1$ and $v_2$, it cannot be a P-cube or a P-matroid USO, either.
\end{proof}

\end{document}